\documentclass{amsart}
\usepackage{amsfonts}

\newtheorem{theorem}{Theorem}[section]

\theoremstyle{definition}
\newtheorem{definition}[theorem]{Definition}

\theoremstyle{remark}
\newtheorem{remark}[theorem]{Remark}

\numberwithin{equation}{section}

\begin{document}

\title[Distance comparison principle]{Distance comparison principle and Grayson type theorem in the three dimensional curve shortening flow}

\author{Siming He}
\address{Department of Mathematics, University of Maryland, College Park, College Park, MD, 20742 }
\email{simhe@math.umd.edu; simhe@zju.edu.cn}


\date{}


\begin{abstract}In this paper, we use the distance comparison
principle, first been developed by G. Huisken, to study the spatial
curve shortening flow. We have got the result that if the initial
curve is the helix, then the local minimum of the ratio of the
extrinsic and intrinsic distance is non-decreasing. And we have
proved a Grayson type theorem for a variety of spatial curves.
\end{abstract}

\maketitle

\section{Introduction}

\label{intro}Let $\gamma_0:\Gamma^1\rightarrow M^m$ be a smooth immersion from a 1-dimensional Remannian manifold $\Gamma$ without boundary to an m-dimensional Riemannian manifold $M$. Consider a one-parameter family of smooth immersions $\gamma:\Gamma\times[0,T)\rightarrow M$ satisfying :
\begin{eqnarray}
\label{CSF}\left\{
\begin{array}{ll}
\frac{\partial}{\partial t}\gamma=k\cdot \vec N,\\
\gamma(\cdot,t)=\gamma_0,
\end{array}\right.
\end{eqnarray}
where $k$ is the geodesic curvature of the curve $\gamma_t(\Gamma)$ and $\gamma_t(x)=\gamma(x,t)$. $\vec N$ is the normal vector of the curve $\gamma_t$. In three dimensional Euclidean spaces, $\vec N$ is the principal normal vector of the spatial curve. We call $\gamma:\Gamma\times[0,T)\rightarrow M$ the curve shortening flow.

If $M^m=\mathbf{R}^2$,then (\ref{CSF}) is called the planar curve shortening flow. If $M^m=\mathbf{R}^3$, (\ref{CSF}) is the spatial curve shortening flow.

The planar curve shortening flow has been thoroughly studied in the last thirty years. The asymptotic behavior of the solutions and the classification of the singularities have attracted much attention.

The following theorems concerning the asymptotic behaviors are famous:

\begin{theorem}(Grayson \cite{Gr87})If the initial curve is closed and imbedded, the solution of the planar curve shortening flow (\ref{CSF}) must become convex before it reaches the maximum existing time T.
\end{theorem}

\begin{theorem}(Hamilton-Gage \cite{G&H86})If the initial curve is closed, convex and imbedded, the solution of (\ref{CSF}) remains convex in the process, and the solution converges to a round point.
\end{theorem}

Before we introduce the results in classification of singularities, we need the following definition:

\begin{definition}
(Type I singularity) If $\lim_{t\rightarrow T}\parallel k^2\left (\cdot,t\right )\parallel_{\infty}\left ( T-t\right ) $ is bounded, we call the singularity formed a Type I singularity;

(Type II singularity) If $\lim_{t\rightarrow T}\parallel k^2\left (\cdot,t\right )\parallel_{\infty}\left ( T-t\right ) $ is unbounded, we call the singularity formed a Type II singularity.
\end{definition}

Then we come to the complete classification of singularities. Notice that this classification works for both planar and spatial curve shortening flows:

\begin{theorem}(\cite{Hui90},\cite{Alt91})

(Type I singularity) If $\lim_{t\rightarrow T}\parallel k^2\left (\cdot,t\right )\parallel_{\infty}\left ( T-t\right ) $ is bounded, then $\gamma$ is asymptotic to a planar solution which is moving by homothety. These planar solutions are given by Abresch and Langer;

(Type II singularity) If $\lim_{t\rightarrow T}\parallel k^2\left (\cdot,t\right )\parallel_{\infty}\left ( T-t\right ) $is unbounded, then there exists a sequence of points and times ${p_n,t_n}$ on which the curvature blows up such that:

(1) a rescaling of the solution along this sequence converges in $C^{\infty}$ to a planar, convex limiting solution $\gamma_\infty$;

(2) $\gamma_\infty$ is a solution which moves by translation called the Grim Reaper.
\end{theorem}

In this paper, we want to find the Grayson type theorem for the spatial curve shortening flows. Namely, we want to answer the following question: under what condition can we ensure that the singularities formed during the spatial curve shortening flow are round points?

There are three possible methods to solve such kind of problems:

(1) The Isoperimetric argument introduced by R. Hamilton (\cite{C&Z01}). However, we could not effectively apply this tool to the spatial curve shortening flow, because we can hardly calculate the area enclosed by the curve and the Sturm Oscillation Theorem, which is vital in the whole argument, can not be directly applied in the three dimensional case;

(2) Huisken's distance comparison principle. In the planar case, based on the complete classification of singularities, Huisken studied the development of the ratio of extrinsic distance and intrinsic distance, and gave us an easy proof of the Grayson theorem (\cite{Hui98}). This method is comparatively better in the three dimensional case;

(3) The techniques developed by Ben Andrews (\cite{And&Ba}) and G. Huisken (\cite{Hui84}) in their research of Mean Curvature Flows. These methods face obstacles in the lower dimensional case because the Myers theorem can not be true on a spatial curve. The Helix is the counter-example.

In a word, Huisken's distance comparison principle is better in the spatial curve shortening flow study.

Now let us give the strategy to find the Grayson type theorem in spatial curve shortening flows.
First of all, we have already understood the classification of all the singularities formed in the flow. The type I singularities can only be the Abresch-Langer Curves.(\cite{Alt91}) If we pose the following condition on the initial curve $\gamma_0$:

\begin{eqnarray}
\label{type I condition}
    \int_{\gamma_0}|k|ds<4\pi,
\end{eqnarray}

we will get the result that all the type I singularities formed must be round points. Because the absolute total curvature is non-increasing in the curve shortening flow(\cite{Alt91}), the type I singularity formed during the flow cannot be other types of Abresch-Langer Curves (they have absolute total curvature greater or equal to $4\pi$ ). So we will have the following theorem:

\begin{theorem}
If the initial curve satisfies (\ref{type I condition}), and the singularity formed in the spatial curve shortening flow is of type I, then the curve shrinks to a round point when approaching the maximum existing time T.
\end{theorem}

\begin{remark}
Because the spatial curve might not remain imbedded during the curve shortening flow (\cite{Alt91}), we cannot use the classical argument, which takes place in the planar curve shortening flow study, to give our result.
\end{remark}

Secondly, we must use certain methods to rule out the possibility of emerging type II singularities. However, when we tried to do that, difficulties happens. The Huisken's distance comparison argument might come into trouble. We discovered that without more information about the shape of the spatial curve, one could not get the result we desired. Our observation based on the analysis of the evolution of the helix under the curve shortening flow.

our main theorem I is as follows:

\begin{theorem}\label{main theorem 1}
Given a curve shortening flow (\ref{CSF}) with $\gamma_0(u)=(a\cos(u),a\sin(u),bu)$ as its initial value. Let $p,q\in\gamma_{t_0}$ be the two points at which the ratio d/l (d denotes the extrinsic metric, l denotes the intrinsic metric) reaches its local minimum,$\forall t_0\in[0,T)$. We always have that:

\begin{eqnarray}
\frac{d}{dt}{\left ( \frac{d}{l} \right )}\ge0
\end{eqnarray}
at $(p,q)$.
\end{theorem}

\begin{remark}
We can show that, the classical distance comparison argument cannot give us this result.
\end{remark}

In spite of the difficulties encountered, we can prove the following main theorem II:

\begin{theorem}\label{main theorem 2}
If the initial curve $\gamma_0$ is an embedded closed curve on $S^2(0)$, the solution of the spatial curve shortening flow (\ref{CSF}) converges to a round point.
\end{theorem}

\begin{remark}
This is the Grayson type theorem we are looking for.
\end{remark}

This result, which was finished at May 28th, 2012, will form part of the author¡¯s undergraduate thesis at Zhejiang University, People's Republic of China. He wishes to thank his mentor, Professor Hongwei Xu for his guidance in the author's thesis work. He also wants to thank Entao Zhao for many helpful discussions and insights.

\section{Proof of Main Theorem I}

First, we introduce some notation:

\begin{definition}
The intrinsic distance $d$ and extrinsic distance $l$ are defined as follows:
\begin{equation*}
\begin{split}
&d,l:\gamma\times\gamma\times[0,T]\rightarrow \mathbf{R}\\
&d(p,q,t):=\left | \gamma(p,t)-\gamma(q,t)\right |_{\mathbf{R}^3} \\
&l(p,q,t):=\int_p^q ds_t
\end{split}
\end{equation*}
\end{definition}

Now we can give the following theorem:

\begin{theorem}
 Given a curve shortening flow (\ref{CSF}) with a complete spatial curve as its initial value, and let $p,q\in \gamma_{t_0}$ be the two points at which the ratio d/l reahces its local minimum, $\forall t_0\in [0,T)$. If either of the folowing conditions are satisfied:
 \begin{eqnarray}
  -\left |e_1+e_2 \right| ^2+\langle e_1+e_2,\omega\rangle^2+\frac{d^2}{l^2}\left ( \int_{\gamma_{t_0}}kds_{t_0}\right )^2\ge 0;
  \end{eqnarray}
 \begin{eqnarray}
  e_1=e_2;
  \end{eqnarray}
  then we have that:
  \begin{eqnarray}
  \frac{d}{dt}{\left ( \frac{d}{l} \right )}\ge0
  \end{eqnarray}
at $(p,q)$.
  Here, $e_1,e_2$ denote the unit tangent vectors of the spatial curve $\gamma_{t_0}$ at points $p,q$. $\omega$ is the direction vector of $\overrightarrow{pq} $.
\end{theorem}

\begin{proof}

Because the ratio $d/l$ reaches its global maximum at the diagonal $\gamma\times \gamma$, without loss of generality, we suppose $p\neq q$ and $s(q)>s(p)$ at $t_0$. Due to the assumption that $d/l$ reaches its local minimum at $(p,q)$, we have the following:

\begin{eqnarray}
\label{local minimum condition}
\delta(\xi)(d/l)(p,q,t_0)=0, \text{   }\delta^2(\xi) (d/l)(p,q,t_0)\ge0
\end{eqnarray}

$\delta(\xi)$ and $\delta^2(\xi)$ are the first and second variation of $d/l$ with respect to the vector $\xi\in T_p\gamma_{t_0}\oplus T_q\gamma_{t_0}$. At the same time, we define:

\begin{equation*}
\begin{split}
&e_1:=\frac{d}{ds}\gamma(p,t_0) \\ &e_2:=\frac{d}{ds}\gamma(q,t_0)\\ &\omega:=d^{-1}(p,q,t_0)(\gamma(q,t_0)-\gamma(p,t_0))
\end{split}
\end{equation*}

With the local minimum condition (\ref{local minimum condition}), we have that $\delta(\xi)=0$ is true for $\forall \delta$. We first choose $\xi=e_1\oplus 0$, then:

\begin{eqnarray}
0=\delta(e_1\oplus 0)(d/l)=\frac{d}{l^2}-\frac{1}{l}\langle \omega,e_1\rangle
\end{eqnarray}

That is to say,

\begin{eqnarray}
\label{first var cond. 1}\langle \omega,e_1\rangle=d/l
\end{eqnarray}

Similarly, we can get:

\begin{eqnarray}
\label{first var cond. 2}\langle \omega,e_2\rangle=d/l
\end{eqnarray}

Now there are two possible situations emerging:

(1) If $e_1=e_2$, then we choose $\xi=e_1\oplus e_2$, and calculate the second variation. Noticing that $\delta(\xi)=0$, we have:

\begin{equation}
\begin{split}
&0\le \delta^2 (e_1\oplus e_2)(d/l)=\\
&\frac{1}{l}\langle\vec k(p,t_0)-\vec k(q,t_0),-\omega\rangle
+\frac{1}{ld}\left | e_1+e_2\right |^2 +\frac{1}{l^2d}\langle e_1+e_2,\omega\rangle\left ( l\langle e_1+e_2,-\omega\rangle-2d\right )+\\
&\frac{2}{l^2}\langle e_1+e_2,-\omega\rangle-\frac{1d}{l^3}(-2)
\end{split}
\end{equation}

It can be easily got that:

\begin{equation}
\label{2.6} 0\le \frac{1}{l}\langle \omega, \vec k(q,t_0)-\vec k(p,t_0) \rangle
\end{equation}

(2) If $e_1\neq e_2$, choose $\xi=e_1\ominus e_2$. Noticing that $\delta(e_1\ominus e_2)(l)=-2$, we have:

\begin{equation}
0=\delta (e_1\ominus e_2)(d/l)=\frac{2d}{l^2}-\frac{1}{l}\langle \omega, e_1+e_2\rangle
\end{equation}

With the condition (\ref{first var cond. 1}) and (\ref{first var cond. 2}), we will have:

\begin{eqnarray}
\label{2.7}\begin{split}
0&\le \delta^2(e_1\ominus e_2)(d/l)\\
&=\frac{1}{l}\langle \omega, \vec k(q,t_0)-\vec k(p,t_0)\rangle+\frac{1}{dl}\left | e_1+e_2 \right |^2-
\frac{1}{dl}\langle \omega,e_1+e_2\rangle ^2
\end{split}
\end{eqnarray}

In contrast with the planar case, vector $e_1+e_2$ is no longer parallel to $\omega$, so the last two terms do not vanish.

With the evolution equation (\ref{CSF}) and conditions (\ref{first var cond. 1}) and (\ref{first var cond. 2}), calculate the derivative of $d/l$ at $(p,q)$ with respect to the time $t$:

\begin{equation}
\label{2.8}\frac{d}{dt}(d/l)(p,q,t_0)=\frac{1}{l}\langle \omega,\vec k(q,t_0)-\vec k(p,t_0)\rangle -\frac{d}{l^2}\frac{d}{dt}(l)
\end{equation}

Due to a result in \cite{A&G92}, we have the following evolution equation:

\begin{eqnarray}
\label{l development}\frac{d}{dt}l=-\int_{\gamma_{t_0}}k^2ds_{t_0}
\end{eqnarray}

Apply (\ref{l development}) and the $H\ddot{o}lder$ inequality, we can convert (\ref{2.8}) into:

\begin{eqnarray}
\label{2.10} \frac{d}{dt}(d/l)(p,q,t_0)\ge\frac{1}{l}\langle \omega,\vec k(q,t_0)-\vec k(p,t_0)\rangle+\frac{d}{l^3}(\int_{\gamma_{t_0}}kds_{t_0})^2
\end{eqnarray}

Now we discuss the following two possible situations:

(1) If $e_1=e_2$, with (\ref{2.6}), (\ref{2.10}), we will have:

\begin{equation*}
\frac{d}{dt}(d/l)(p,q,t_0)\ge \frac{d}{l^3}(\int_{\gamma{t_0}}kds_{t_0})^2\ge 0
\end{equation*}

(2) If $e_1\neq e_2$, according to the condition given and (\ref{2.7}), we will get:

\begin{equation*}
\frac{d}{dt}(d/l)(p,q,t_0)\ge \frac{1}{ld}\left (- \left | e_1+e_2 \right |^2+\langle e_1+e_2,\omega\rangle^2+\frac{d^2}{l^2}\left (\int k ds_{t_0}\right )^2 \right )\ge 0
\end{equation*}

Now the proof is complete.
\end{proof}

\begin{remark}
The conditions of the theorem seem unnatural. However, with the following observation (next remark) and our main theorem, one will see that, unless having  more concise understanding of the shape of the spatial curve between the two points $(p,q)$ where $d/l$ reaches its local minimum, the Huisken's distance comparison argument will not work in the spatial Curve Shortening Flow. That's why we give the conditions as in the theorem stated.
\end{remark}

\begin{remark}
For the helix $\gamma(u,t_0)=(a\cos(u),a\sin(u),bu)$, we can write down every geometric value:

The unit tangent vector:

\begin{equation}
\vec e=\frac{(-a\sin(u),a\cos(u),b)}{\sqrt{a^2+b^2} }
\end{equation}

The curvature vector:

\begin{equation}
\vec k=k\cdot \vec N=\frac{(-a\cos(u),-a\sin(u),0)}{a^2+b^2}
\end{equation}

The curvature and torsion:

\begin{eqnarray}
\begin{split}
k=\frac{a}{a^2+b^2} \\
\tau =\frac{b}{a^2+b^2}
\end{split}
\end{eqnarray}

Now we can calculate $\frac{d}{dt}(d/l)$ of the helix directly. Note that, the curvature of helix is constant, so the $\ge$ in (\ref{2.10}) is actually $=$. We have:

\begin{eqnarray}
\label{2.11}\frac{d}{dt}(d/l)(p,q,t_0)=\frac{1}{l}\langle \omega,\vec k(q,t_0)-\vec k(p,t_0)\rangle+\frac{d}{l^3}(\int_{\gamma_{t_0}}kds_{t_0})^2
\end{eqnarray}

If we just apply the local minimum conditon (\ref{local minimum condition}), we can only use (\ref{2.7}) to simplify (\ref{2.11}), giving:

\begin{equation}
\frac{d}{dt}(d/l)(p,q,t_0)\ge \frac{1}{ld}\left (- \left | e_1+e_2 \right |^2+\langle e_1+e_2,\omega\rangle^2+\frac{d^2}{l^2}\left (\int k ds_{t_0}\right )^2 \right )
\end{equation}

We now focus on the function $F=- \left | e_1+e_2 \right |^2+\langle e_1+e_2,\omega\rangle^2+\frac{d^2}{l^2}\left (\int k ds_{t_0}\right )^2$. Under the assumption of being a helix, we can calculate it directly:

\begin{equation}
F=-4+\frac{(2-2\cos(y))}{1+m}+\frac{4m}{1+m}+\frac{4(2-2\cos(y))}{(1+m)(y)^2} +\frac{2-2\cos(y)+m(y)^2}{(1+m)^2}
\end{equation}

Here, $y=t_1-t_2$, $m=b^2/a^2$.

After using Matlab to calculate the value, we find that when $m$ is small, the function $F$ can be negative. Therefore, unless more information other than the local minimum condition (\ref{local minimum condition}) are given, the Huisken's argument cannot guarantee that $(d/l)_{min}$ is non-decreasing.
\end{remark}

Without using any shape information, we can just get the following result:

\begin{theorem}
Given a curve shortening flow with the helix $\gamma(u,0)=(a\cos(u),$$a\sin(u),bu)$ as its initial value. Let $p,q\in\gamma_{t_0}$ be the points where d/l reaches its local minimum, $\forall t_0\in[0,T)$. If $b^2/a^2=(\tau/k)^2$ is large enough, then we will have that:

\begin{equation}
   \frac{d}{dt}{\left ( \frac{d}{l} \right )}\ge0
\end{equation}
at $(p,q)$.
\end{theorem}

\begin{proof}
As in the discussions above, we have:

\begin{equation*}
\label{x}
\begin{split}
&\frac{d}{dt}(d/l)(p,q,t_0)\ge \frac{1}{ld}F =\\
&\frac{1}{ld(1+m)^2}  \{-4(1+m)^2+(2-2\cos(y))(1+m)+4m(1+m)+\\
&4(1+m)(2-2\cos(y))(y)^{-2}
+2-2\cos(y)+m(y)^2  \}
\end{split}
\end{equation*}

We now focus on the following function:

\begin{equation*}
\begin{split}
G:=&-4(1+m)^2+(1+m)(2-2\cos(y))+4m(1+m)+\\
&(1+m)\frac{4(2-2\cos(y))}{(y)^2} +2-2\cos(y)+m(y)^2
\end{split}
\end{equation*}

The Taylor expansion of $\cos(y)$ is:

\begin{equation*}
\cos(y)=\sum^\infty_{n=0}(-1)^n\frac{y^{2n}}{(2n)!},x\in(-\infty,+\infty) \end{equation*}

We expand $G$ using the above expansion:

\begin{equation*}
\begin{split}
G&\ge -4(1+m)^2+4m(1+m)+\frac{4(1+m)(y^2-2\sum^\infty_{n=2} (-1)^n \frac {y^{2n}}{(2n)!}} {y^2}+my^2 \\ &=-4-8m-4m^2+4m+4m^2+4+4m-8(1+m)\sum^\infty_{n=2}(-1)^n\frac{y^{2(n-1)}}{(2n)!}+my^2 \\
&=my^2-8(1+m)\sum^\infty_{n=2}(-1)^n\frac{y^{2(n-1)}}{(2n)!}\\
&=\left (m-\frac{1+m}{3} \right) y^2+8(1+m)\sum^\infty_{n=2}(-1)^n\frac{y^{2n}}{(2n+2)!}
\end{split}
\end{equation*}

For $m$ large enough, the above will be non-negative for all $y$. Namely, $\frac{d}{dt}(\frac{d}{l})\ge0$. The result follows.
\end{proof}

However, we could no longer get further due to the obstacles we gave above. However, if we have noticed the special shape of the helix, we can get a direct proof of the main theorem I.

\begin{proof}
If we calculate directly from equation (\ref{2.11}), we get:

\begin{equation*}
\frac{d}{dt}(d/l)(p,q,t_0)=\frac{1}{l}\langle \omega,\vec k(q,t_0)-\vec k(p,t_0)\rangle+\frac{d}{l^3}(\int_{\gamma_{t_0}}kds_{t_0})^2
\end{equation*}

And as mentioned before, we have:

\begin{equation*}
\vec k=k\cdot \vec N=\frac{(-a\cos(u),-a\sin(u),0)}{(a^2+b^2)}
\end{equation*}

Now we can directly calculate $\frac{d}{dt}(\frac{d}{l})$:

\begin{equation*}
\begin{split}
&\frac{d}{dt}(d/l)(p,q,t_0)=\frac{1}{ld}\left(\langle d\omega, \vec k (q,t_0)-\vec k (p,t_0)\rangle +\frac{d^2}{l^2}\left ( \int_{\gamma_0}|k|ds \right )^2\right )\\
&=\frac{1}{ld}\left (\frac{a}{a^2+b^2}\langle F(q,t_0)-F(p,t_0),(\cos u_1-\cos u_2,\sin u_1-\sin u_2,0)\rangle +d^2\left ( \frac {a}{a^2+b^2}\right )^2\right )\\
&=\frac{1}{ld(1+m)}\frac{a^2}{a^2+b^2}\left(-2(1+m)+2\cos y (1+m) +2-2\cos y+{m}y^2\right )\\
&=\frac{2m}{ld(1+m)}\frac{a^2}{a^2+b^2}\left (\cos y-(1-\frac{1}{2}y^2) \right )\\
&\ge 0
\end{split}
\end{equation*}

The result now follows.

\end{proof}

Now we consider a more general case:

\begin{eqnarray}\label{2.15}
\gamma(u,t_0)=(f(u),g(u),bu)
\end{eqnarray}

To make the following calculation easier, we give the following condition:

\begin{eqnarray}
\label{A1,A2}
\begin{split}
f'(u)^2+g'(u)^2=a^2 \\
f''(u^2)+g''(u)^2\le A
\end{split}
\end{eqnarray}

Here, $a,A$ are constants.
Now we give the following theorem concerning the curves we just stated:

\begin{theorem}
Suppose at $t_0$ the solution $\gamma(u,t_0)$ of the spatial curve shortening flow is the curve in (\ref{2.15}), and the conditions (\ref{A1,A2}) are satisfied. Furthermore, if at the points $p,q$ where $d/l$ reaches its local minimum, we have:

\begin{equation*}
\begin{split}
&(f''(u_2)-f''(u_1))(f(u_2)-f(u_1))+(g''(u_2)-g''(u_1))(g(u_2)-g(u_1))\\
&+\frac{A}{a^2+b^2}((f(u1)-f(u_2))^2+((g(u_1)-g(u_2))^2+b^2(u_1-u_2)^2)\\
&\ge 0
\end{split}
\end{equation*}

Then we have that:

\begin{equation*}
   \frac{d}{dt}{\left ( \frac{d}{l} \right )}\ge0
\end{equation*}
at $(p,q)$.
\end{theorem}

\begin{proof}
We can directly calculate every geometric values:

\begin{equation*}
\begin{split}
&\vec T=(f'(u),g'(u),b)/\sqrt{a^2+b^2}\\
&k\cdot \vec N=(f''(u),g''(u),0)/(a^2+b^2)\\
&\omega=\frac{(f(u_2)-f(u_1),g(u_2)-g(u_1),b(t_2-t_1))}{\sqrt{(f(u_1)-f(u_2))^2+(g(u_2)-g(u_1))^2+b^2(u_1-u_2)^2}}
\end{split}
\end{equation*}

With the calculation in (\ref{2.8})(\ref{l development}), we have the following:

\begin{eqnarray}
\frac{d}{dt}(d/l)(p,q,t_0)=\frac{1}{l}\langle \omega,\vec k(q,t_0)-\vec k(p,t_0)\rangle +\frac{d}{l^2}\int k^2 ds_{t_0}
\end{eqnarray}

Insert all the values calculated, we have:
\begin{equation*}
\begin{split}
&\frac{d}{dt}(d/l)(p,q,t_0)\\
=&\frac{1}{l}\langle \omega,\vec k(q,t_0)-\vec k(p,t_0)\rangle +\frac{d}{l^2}\int k^2 ds_{t_0} \\
=&l^{-1}\frac{(f''(u_2)-f''(u_1))(f(u_2)-f(u_1))+(g''(u_2)-g''(u_1))(g(u_2)-g(u_1))} {\sqrt{(f(u1)-f(u_2))^2+((g(u_1)-g(u_2))^2+b^2(u_1-u_2)^2}(a^2+b^2)}\\
&+\frac{\sqrt{((f(u1)-f(u_2))^2+((g(u_1)-g(u_2))^2+b^2(u_1-u_2)^2}} {\sqrt{a^2+b^2}(t_2-t_1)} \int_{t_1}^{t_2}\frac{f''(u)^2+g''(u)^2}{(a^2+b^2)^{3/2}}du\\
=&\frac{F}{l(a^2+b^2)n}
\end{split}
\end{equation*}

Here,
\begin{equation*}
\begin{split}
n=&\sqrt{((f(u1)-f(u_2))^2+((g(u_1)-g(u_2))^2+b^2(u_1-u_2)^2}\\
F=&(f''(u_2)-f''(u_1))(f(u_2)-f(u_1))+(g''(u_2)-g''(u_1))(g(u_2)-g(u_1))\\ &+\frac{A}{a^2+b^2}((f(u1)-f(u_2))^2+((g(u_1)-g(u_2))^2+b^2(u_1-u_2)^2)
\end{split}
\end{equation*}

According to the condition given, we have $F\ge 0$, so $\frac{d}{dt}(\frac{d}{l})\ge 0$.
\end{proof}

\begin{remark}
It can be easily checked that the helix satisfy the conditions in the theorem.
\end{remark}

\section{Some discussion of the spatial closed curves}

Similar to the calculation in \cite{Hui98}, we can give theorems concerning the solutions of (\ref{CSF}) with a spatial closed curve as its initial value.

First, we need a definition:

\begin{definition}
Define $L:=\int_{\gamma_{t_0}}ds_{t_0}$ to be the total length of the curve $\gamma_{t_0}$. The intrinsic distance $l$ can only be smoothly defined for $0\le l<L/2$. We now define the following quantity as in \cite{Hui98}:

\begin{equation*}
\psi:=\frac{L}{\pi}\sin{\frac{l\pi}{L}}
\end{equation*}

And we define:

\begin{equation*}
 \alpha:=(l\pi/L)
\end{equation*}
\end{definition}

\begin{theorem}
Suppose $\gamma$ is the solution of the (\ref{CSF}). $d/{\psi} $ reaches its minimum at $(p,q)\in\gamma_{t_0}$. If at $(p,q)$, either of the following conditions is satisfied:

\begin{eqnarray}\label{3.1}
\cos {\alpha}\left (\int_p^q|k|ds_{t_0}\right )^2-\cos{\alpha}\frac{4\pi^2l^2}{L^2}-\frac{\psi l |e_1+e_2|^2}{d^2}+\frac{4l}{\psi}\cos^2 \alpha\ge 0
\end{eqnarray}
or:

\begin{eqnarray}
e_1=e_2
\end{eqnarray}
then we will have that:

\begin{equation*}
\frac{d}{dt}\left (\frac{d}{\psi}\right )\ge 0
\end{equation*}
at $(p,q)$.
\end{theorem}

\begin{proof}
Because $d/\psi$ reaches local minimum at $(p,q)$, so we have:

\begin{equation*}
\begin{split}
\delta(d/\psi)=0 \\
\delta^2(d/\psi)\ge0
\end{split}
\end{equation*}

Let the first variation vector be $\xi=e_1 \oplus  0$, then we have:

\begin{equation*}
\begin{split}
0 &=\delta(e_1\oplus 0)(\frac{d}{\psi})(p,q,t_0) \\
&=\psi^{-2}\left (\frac{\psi}{d}\langle e_1,\gamma(p,t_0)-\gamma(q,t_0)\rangle+d\cos \frac{l\pi}{L} \right )\\
&=\psi^{-2} \left ( \psi\langle e_1,-\omega\rangle+d\cos\left ( \frac{l\pi}{L}\right )\right )
\end{split}
\end{equation*}

So we will have the equation:

\begin{eqnarray}\label{fir. var. condition 21}
\langle \omega, e_1\rangle=\frac{d}{\psi}\cos\left ( \frac{l\pi}{L}\right )
\end{eqnarray}

Similarly, one can get:

\begin{eqnarray}\label{fir. var. condition 22}
\langle \omega, e_2\rangle=\frac{d}{\psi}\cos\left ( \frac{l\pi}{L}\right )
\end{eqnarray}

Now we have two possibilities:

(1) If $e_1=e_2$, let the variation vector be $\xi=e_1\oplus e_2$, then:

\begin{equation*}
\begin{split}
0\le&\delta^2(\xi)(d/\psi)(p,q,t_0) \\
=&\frac{1}{\psi d} \langle \frac{d^2\gamma(p,t_0)}{d s_{t_0}^2}-\frac{d^2\gamma(q,t_0)}{d s_{t_0}^2},\gamma(p,t_0)-\gamma(q,t_0)\rangle+\frac{1}{\psi d} \parallel\frac{d\gamma(p,t_0)}{ds_{t_0}}-\frac{d \gamma(q,t_0)}{ds_{t_0}}\parallel^2 \\
&-\frac{1}{(\psi d)^2}\left\langle \frac{d\gamma(p,t_0)}{ds_{t_0}}-\frac{d\gamma(q,t_0)}{ds_{t_0}},\gamma(p,t_0)-\gamma(q,t_0)\right\rangle \delta(\xi)(\psi d)\\
=&\frac{1}{\psi d}\langle \vec k(p,t_0)-\vec k(q,t_0),\gamma(p,t_0)-\gamma (q,t_0)\rangle\\
=&\frac{1}{\psi}\langle \vec k(p,t_0)-\vec k(q,t_0),\omega\rangle
\end{split}
\end{equation*}

On the second last line above, we use the following fact:

\begin{equation*}
\frac{d\gamma(p,t_0)}{ds_{t_0}}-\frac{d\gamma(q,t_0)}{ds_{t_0}}=e_1-e_2=0 \end{equation*}

Therefore, we have:

\begin{eqnarray}\label{3.4}
\langle \vec k(p,t_0)-\vec k(q,t_0),\omega\rangle\ge 0
\end{eqnarray}

(2) If $e_1\neq e_2$, let the variation vector be $\xi=e_1\ominus e_2 $. Noticing that:

\begin{equation*}
\delta(\xi)(\psi)=\frac{L}{\pi}\cos(\frac{l\pi}{L})\frac{\pi}{L} \left (\delta(\xi)(l)\right )=-2\cos(\frac{l\pi}{L})
\end{equation*}

we then have:

\begin{equation*}
\begin{split}
0\le& \delta^2(\xi)(d/\psi)(p,q,t_0)\\
=&\delta(\xi)\left ( \frac{\psi \delta(\xi)(d)-d\delta(\xi)(\psi)}{\psi^2}\right )\\
=& \frac{1}{\psi}\langle \vec k(q,t_0)-\vec k(p,t_0),\omega\rangle +\frac{1}{\psi d}|e_1+e_2|^2-\frac{1}{\psi d}\langle e_1+e_2,\omega\rangle^2+\frac{4d\pi^2}{\psi L^2}
\end{split}
\end{equation*}

Here we have used the conditions (\ref{fir. var. condition 21}) and (\ref{fir. var. condition 22}).

Now we are ready to compute $\frac{d}{dt}(d/\psi)$:

\begin{equation*}
\begin{split}
\frac{d}{dt}(\frac{d}{\psi})=& \frac{1}{\psi d}\left \langle \gamma(q,t_0)-\gamma(p,t_0),\vec k(q,t_0)-\vec k(p,t_0)\right \rangle-\frac{d}{\psi^2}\frac{d}{dt}\left (\frac{L}{\pi}\sin(\frac{l\pi}{L}\right ) )\\
=&\frac{1}{\psi}\langle \omega,\vec k(q,t_0)-\vec k(p,t_0)\rangle+\frac{d}{\psi^2\pi}\sin(\alpha)\int_{S^1}k^2ds_{t_0}+\\
&\frac{d}{\psi^2}\cos(\alpha)\int_p^q k^2ds_{t_0}-\frac{dl}{\psi^2L}\cos(\alpha)\int_{S^1}k^2ds_{t_0}
\end{split}
\end{equation*}

Here we have used the condition $\frac{d}{dt}(ds_{t_0})=-k^2ds_{t_0}$. Now let us discuss all the possibilities:

(1) If $e_1=e_2$, first notice that $\frac{d}{\psi^2\pi}\sin(\alpha)=\frac{d}{\psi L}$. Together with the calculation above we have that:

\begin{equation*}
\begin{split}
\frac{d}{dt}(\frac{d}{\psi})\ge& \frac{d}{\psi^2\pi}\sin(\alpha)\int_{S^1}k^2ds_{t_0}+\frac{d}{\psi^2}\cos(\alpha)\int_p^q k^2ds_{t_0}-\frac{dl}{\psi^2L}\cos(\alpha)\int_{S^1}k^2ds_{t_0} \\
\ge & \frac{d}{\psi L}(1-\frac{l}{\psi}\cos \alpha)\int_{S^1}k^2ds_{t_0}+\frac{d}{\psi^2}\cos(\alpha)\int_p^qk^2ds_{t_0}
\end{split}
\end{equation*}

Noticing that $\frac{l}{\psi}\cos(\alpha)=(l\pi/L)/\tan(l\pi/L)=\alpha/\tan(\alpha)\le1 $, we have $\frac{d}{dt}\left (\frac{d}{\psi}\right )\ge 0$.

(2) If $e_1\neq e_2$, with Fenchel theorem, we have $\int_\gamma|\vec k|ds\ge 2\pi $, so we will get $\int_\gamma k^2 ds\ge \frac{r\pi^2}{L} $. Together with the calculation above, we have:

\begin{equation*}
\begin{split}
\frac{d}{dt}(\frac{d}{\psi})\ge&-\frac{1}{\psi d}\left (|e_1+e_2|^2-\langle e_1+e_2,\omega\rangle^2\right )-\frac{4d\pi^2}{\psi L^2} +\frac{d}{\psi L}\left ( 1-\frac{l}{\psi}\cos(\alpha)\right )\int_{S^1}k^2ds_{t_0}\\
&+\frac{d}{\psi^2}\cos(\alpha)\int_p^qk^2ds_{t_0}\\
\ge& -\frac{1}{\psi d}(|e_1+e_2|^2-\langle e_1+e_2,\omega\rangle^2)-\frac{4d\pi^2}{\psi L^2}+\frac{4d\pi^2}{\psi L^2}(1-\frac{l}{\psi}\cos(\alpha))\\
&+\frac{d}{\psi^2}\cos(\alpha)\int_p^qk^2ds_{t_0}\\
=&\frac{d}{\psi^2 l}\left (\cos(\alpha)l\int_p^qk^2ds_{t_0}-\cos(\alpha)\frac{4\pi^2 l^2}{L^2}-\frac{\psi l|e_1+e_2|^2}{d^2}+\frac{4l\cos^2(\alpha)}{\psi}\right )\\
\ge&\frac{d}{\psi^2 l}\left (\cos {\alpha}\left (\int_p^qkds_{t_0}\right )^2-\cos(\alpha)\frac{4\pi^2 l^2}{L^2}-\frac{\psi l|e_1+e_2|^2}{d^2}+\frac{4l}{\psi}\cos^2(\alpha)\right )
\end{split}
\end{equation*}

On the second last line we used the conditions (\ref{fir. var. condition 21}) and (\ref{fir. var. condition 22}). On the last line, we have used the $H\ddot{o}lder $ inequality.

Together with the condition of the theorem, we have $\frac{d}{dt}(\frac{d}{\psi})\ge 0 $.

\end{proof}

\begin{remark}
For some spatial curves, the condition of the above theorem can be satisfied. For example, the closed spatial curve $\gamma_(u,t_0)=(\cos u,\sin u,\cos(2u)) $ satisfies the condition. Because the minimum "pairs" $(p,q)$ take place at the two crests and two valleys, we get:

\begin{equation*}
\cos {\alpha}\left (\int_p^q|k|ds_{t_0}\right )^2-\cos{\alpha}\frac{4\pi^2l^2}{L^2}-\frac{\psi l |e_1+e_2|^2}{d^2}+\frac{4l}{\psi}\cos^2 \alpha=0
\end{equation*}

Because $|e_1+e_2|=0,\cos \alpha=0 $, the whole term becomes zero. Therefore, the condition is satisfied.
\end{remark}

\section{Proof of main theorem II}

First, we restate main theorem II:

\begin{theorem}
If the initial curve $\gamma_0$ is an embedded closed curve on $S^2(0)$, the solution of the spatial curve shortening flow (\ref{CSF}) converges to a round point.
\end{theorem}

\begin{proof}
First, we prove the curve $\gamma(\cdot,t)$ remains on the sphere $S^2_{\sqrt{1-2t}}(0) $.

\begin{equation*}
\begin{split}
\frac{\partial}{\partial t}(|\gamma|^2)=&2\langle \gamma,\frac{\partial}{\partial t}\gamma\rangle=2\langle\gamma,k\cdot \vec N\rangle\\
=&2\langle\gamma,\frac{\partial^2}{\partial s^2}\gamma\rangle\\
=&\frac{\partial^2}{\partial s^2}\langle \gamma,\gamma\rangle-2\langle\frac{\partial}{\partial s}\gamma,\frac{\partial}{\partial s}\gamma\rangle\\
=&\frac{\partial^2}{\partial s^2}(|\gamma|^2)-2
\end{split}
\end{equation*}

Simply speaking,

\begin{equation*}
\frac{\partial}{\partial t}(|\gamma|^2)=\frac{\partial^2}{\partial s^2}(|\gamma|^2)-2
\end{equation*}

Because the solution to these kind of partial differential equations is unique, we get that:

\begin{equation*}
|\gamma|^2=1-2t
\end{equation*}
is the only solution of the above equation. That is equivalent to saying that $\gamma(\cdot,t)$ always stays on the sphere $S^2_{\sqrt{1-2t}}(0) $.

Notice that, if $\gamma(\cdot,t)$ stays on a sphere $S^2_{\sqrt{1-2t}}(0) $, its curvature vector can be decomposed. Precisely speaking, the vector $k\cdot \vec N$ can be written as:

\begin{equation*}
k\cdot \vec N=k_g\vec Q+k_n\vec n
\end{equation*}

Here, $k_g$ denotes the geodesic curvature of $\gamma(\cdot,t)$ on the surface $S^2_{\sqrt{1-2t}}(0) $. $k_n$ denotes the normal curvature of the sphere. $\vec n$ is the inner normal unit vector of the surface and $\vec Q= \vec n \times \vec T$.

We can also write $\gamma(\cdot,t)$ as follows:

\begin{eqnarray}
\label{4.1}
\gamma(\cdot,t)=\sqrt{1-2t}\cdot\tilde{\gamma}(\cdot,t)
\end{eqnarray}

Now, $\tilde{\gamma} $ stays on the unit sphere, namely, $\parallel \tilde{\gamma}\parallel=1 $.

We insert (\ref{4.1}) into the original curve shortening equation, and we will get a new equation:

\begin{eqnarray}
\label{4.2}
\frac{-1}{\sqrt{1-2t}}\tilde{\gamma}+\sqrt{1-2t}\frac{\partial \tilde{\gamma}}{\partial t}=k_g\vec Q+k_n\vec n
\end{eqnarray}

Notice that $\tilde{\gamma} $ and $\vec n$ only differ by a sign. And for the sphere $S^2_{\sqrt{1-2t}}(0) $, $k_n=1/r=\frac{1}{\sqrt{1-2t}} $. Therefore, we have:

\begin{eqnarray}
\label{4.3}
\sqrt{1-2t}\cdot \frac{\partial \tilde{\gamma}}{\partial t}=k_g\cdot \vec Q
\end{eqnarray}

We also notice the following relation:

\begin{eqnarray}
\label{4.4}
\tilde{k_g}=\sqrt{1-2t}\cdot k_g
\end{eqnarray}

Combining (\ref{4.3})(\ref{4.4}), we have that:

\begin{eqnarray}
\label{4.5}
(1-2t)\cdot \frac{\partial}{\partial t}\tilde{\gamma}=\tilde{k_g}\cdot \vec{Q}
\end{eqnarray}

Let us dilate the time:

\begin{equation*}
\tilde{t}=-\frac{1}{2}\log(\frac{1}{2}-t)
\end{equation*}

Because $\gamma(\cdot,t)$ keeps on the sphere $S^2_{\sqrt{1-2t}}(0) $, the maximum existing time $T\le \frac{1}{2} $. Namely, $\log(\frac{1}{2}-t) $ is meaningful for $\forall t\in(0,T) $. We can easily get that:

\begin{eqnarray}
\label{4.7}
\frac{\partial}{\partial \tilde{t}}=(1-2t)\frac{\partial}{\partial t}
\end{eqnarray}

Taking (\ref{4.7}) into (\ref{4.5}), and noticing that $\tilde{\vec Q}=\vec Q $, we will have the following evolution equation for $\tilde{\gamma} $:

\begin{eqnarray}
\frac{\partial }{\partial \tilde{t}}{\tilde{\gamma}}=\tilde{k_g}\cdot\tilde{\vec Q}
\end{eqnarray}

Because $ \parallel \tilde{\gamma}\parallel=1 $, that is a curve shortening flow for a closed curve on $S^2_1(0) $.

Therefore, we can now apply the distance comparison argument in \cite{D&M}, or we can just use X.P.Zhu's theorem (\cite{Zhu98}). Then we have that $\tilde{\gamma} $ converges to a round point in the $C^\infty$ sense. Because $\tilde{\gamma} $ is just a homothety of $\gamma$, $\gamma$ converges to a round point.
\end{proof}

\end{document}